\documentclass[DIV=14,letterpaper]{scrartcl}

\usepackage{tikz}
\usetikzlibrary{calc}
\usetikzlibrary{matrix}
\usepackage{amsmath,amssymb,amsfonts,amsthm}
\usepackage{graphicx}
\allowdisplaybreaks
\usepackage{subfigure,multicol,multirow}
\usepackage{makecell}
\usepackage{hyperref} 
\usepackage{diagbox}

\theoremstyle{plain}

\newtheorem{pro}{\hspace{6mm}Proposition}[section]
\newtheorem{lem}{\hspace{6mm}Lemma}[section]

\theoremstyle{definition}

\theoremstyle{remark}

\newcommand{\V}[1]{\mathbf{#1}}

\newcommand{\email}[1]{\href{mailto:#1}{#1}}

\title{Convergence analysis of iterative solution with inexact block preconditioning for weak Galerkin finite element approximation of Stokes flow}
\author{
Weizhang Huang\thanks{Department of Mathematics, the University of Kansas, 1460 Jayhawk Blvd, Lawrence, KS 66045, USA (\email{whuang@ku.edu}).}
\and
Zhuoran Wang\thanks{Department of Mathematics, the University of Kansas, 1460 Jayhawk Blvd, Lawrence, KS 66045, USA (\email{wangzr@ku.edu}).}
}

\date{} 

\begin{document}

\maketitle

\textbf{Abstract.}
This work is concerned with the convergence of the iterative solution for the Stokes flow,
discretized with the weak Galerkin finite element method and preconditioned using inexact block Schur complement
preconditioning. The resulting saddle point linear system is singular and the pressure solution is not unique.
The system is regularized with a commonly used strategy by specifying the pressure value at a specific location.
It is analytically shown that the regularized system is nonsingular but has an eigenvalue approaching zero as
the fluid kinematic viscosity tends to zero.
Inexact block diagonal and triangular Schur complement preconditioners
are considered with the minimal residual method (MINRES) and the generalized minimal residual method (GMRES), respectively.
For both cases, the bounds are obtained for the eigenvalues of the preconditioned systems and
for the residual of MINRES/GMRES. 
These bounds show that the convergence factor of MINRES/GMRES is almost independent of the viscosity parameter
and mesh size while the number of MINRES/GMRES iterations required to
reach convergence depends on the parameters only logarithmically.
The theoretical findings and effectiveness of the preconditioners are verified
with two- and three-dimensional numerical examples.

\vspace{5pt}


\vspace{5pt}


\section{Introduction}
\label{SEC:intro}
We consider the Stokes flow problem
\begin{equation}
\begin{cases}
\displaystyle
    -\mu \Delta \mathbf{u} + \nabla p  =  \mathbf{f},
      \quad \mbox{in } \; \Omega,
    \\
    \displaystyle
    \nabla \cdot \mathbf{u}  =  0,
     \quad \text{ in } \Omega,
    \\
    \displaystyle
    \mathbf{u}  =  \mathbf{g},
    \quad \mbox{on } \; \partial \Omega,
\end{cases}
\label{Eqn_StokesBVP}
\end{equation}
where $ \Omega \subset \mathbb{R}^d$ $(d=2,3) $ is a bounded polygonal/polyhedral domain,
$ \mu>0 $ is the fluid kinematic viscosity,
$ \mathbf{u} $ is the fluid velocity,
$ p $ is the fluid pressure,
$ \mathbf{f} $ is a body force,
$\mathbf{g}$ is a boundary datum of the velocity satisfying the compatibility condition
$ \int_{\partial \Omega}\mathbf{g}\cdot\mathbf{n}=0 $,
and $\mathbf{n}$ is the unit outward normal to
the boundary $\partial \Omega$ of the domain.
For this problem, the pressure is not unique.
A unique pressure can be obtained by requiring its mean to be zero or by specifying its value at a specific location.

Numerical solution of Stokes flow problems
has continuously gained attention from researchers. Particularly,
a variety of finite element methods have been studied for those problems; e.g., see
\cite{Ainsworth_SINUM_2022} (mixed finite element methods),
\cite{Bevilacqua_SISC_2024,Wang2_CMAME_2021} (virtual element methods),
\cite{Lederer_JSC_2024,TuWangZhang_ETNA_2020} (hybrid discontinuous Galerkin methods),
and \cite{MR3261511,TuWang_CMA_2018,WangYe_Adv_2016} (weak Galerkin (WG) finite element methods).
We use here the lowest-order WG method for the discretization of Stokes flow problems. 
It is known (cf. Lemma~\ref{pro:err1} or \cite{WANG202290})
that the lowest-order WG method, without using stabilization terms, satisfies the inf-sup or LBB
condition (for stability)
and has the optimal-order convergence. Moreover, the error in the velocity is independent of the error in the pressure
(pressure-robustness) and the error in the pressure is independent of the viscosity $\mu$ ($\mu$-semi-robustness).
On the other hand, efficient iterative solution of the saddle point system arising from the WG approximation of Stokes problems
has not been studied so far.

In this work, we are interested in the efficient iterative solution of the saddle point system
resulting from the lowest-order WG discretization of (\ref{Eqn_StokesBVP}).
The system is singular and the pressure solution is not unique.
We employ a commonly used strategy to avoid the singularity and make the pressure solution unique
by specifying the value zero of the pressure at the barycenter of the first mesh element.
This regularization has several advantages such as it maintaining the symmetry and sparseness of
the original system, it not altering the solution, and its implementation being simple and straightforward.
On the other hand, the regularization is local and the regularized system is almost singular, having
an eigenvalue approaching to zero as $\mu \to 0$. This poses challengers in developing efficient
preconditioning and iterative solution methods.

We consider inexact block diagonal and triangular Schur complement preconditioners for the regularized
saddle point system. The block diagonal preconditioner maintains the symmetry of the system and the preconditioned
system can be solved
with the minimal residual method (MINRES). The spectral analysis can also be used to analyze the convergence
of MINRES. On the other hand, the block triangular preconditioners lead to nondiagonalizable preconditioned systems,
for which we need to use the generalized minimal residual method (GMRES) and the spectral analysis is typically
insufficient to determine the convergence of GMRES. To circumvent this difficulty,
two lemmas (Lemmas~\ref{lem:gmres-conv-lower} and \ref{lem:gmres-conv-upper}) are developed
in Appendix~\ref{appendix:gmres convergence} that provide an estimate on the residual of GMRES with block triangular
preconditioners for general saddle point systems through estimating the norm of the off-diagonal
blocks in the preconditioned system and the spectral analysis of the preconditioned Schur complement.
For both situations, bounds for the eigenvalues of the preconditioned systems and for the residual of MINRES/GMRES
are obtained. These bounds show that the convergence factor of MINRES/GMRES is almost independent of $\mu$ and
$h$ (mesh size) and the number of MINRES/GMRES iterations required to reach convergence
depends on these parameters only logarithmically. They also show that GMRES with block triangular preconditioners
requires about half as many iterations as MINRES with block diagonal preconditioners to reach a prescribed
level of the relative residual.

It should be pointed out that the solution of general saddle point systems has been studied extensively
and remains a topic of active research; e.g., see review articles \cite{BenziGolubLiesen-2005,Benzi2008}
and more recent works \cite{Ainsworth_SINUM_2022,Bacuta-2019,Rhebergen_SISC_2022}.
Most systems that have been studied are either singular systems with a single eigenvalue exactly equal to zero
and other eigenvalues away from zero or nonsingular systems with eigenvalues away from zero.
Little work has been done for almost singular systems.
Moreover, limited analysis work has been done with block triangular preconditioners.
Bramble and Pasciak \cite{Bramble_MathComp_1988} considered and gave convergence analysis for
a lower block triangular preconditioner for symmetric saddle point systems (see more detailed discussion on this topic
in Appendix~\ref{appendix:gmres convergence}).

The rest of this paper is organized as follows.
In Section~\ref{SEC:formulation}, the weak formulation for Stokes flow 
and its discretization by the lowest-order WG method are described.
System regularization and the approximations to the Schur complement are studied also in the section. 
Section~\ref{sec::regularization-diagonal} discusses the convergence of MINRES with inexact block diagonal Schur complmenent preconditioning.
The inexact block triangular Schur complement preconditioning and convergence of GMRES
for the regularized system are studied in Section~\ref{sec::regularization-triagle}.
Numerical results in both two and three dimensions are presented in Section~\ref{SEC:numerical}
to verify the theoretical findings and showcase the effectiveness of the preconditioners.
The conclusions are drawn in Section~\ref{SEC:conclusions}.
Finally, Appendix~\ref{appendix:gmres convergence} discusses block Schur complement preconditioners for general saddle point problems. In particular, two lemmas are developed for the convergence of GMRES with block triangular Schur complement preconditioning.


\section{Weak Galerkin discretization and system regularization for Stokes flow}
\label{SEC:formulation}

In this section, we describe the lowest-order WG finite element discretization of the Stokes flow problem (\ref{Eqn_StokesBVP}).
The resulting linear system is regularized with a constraint to ensure the uniqueness of the pressure solution.

We start with the weak formulation of (\ref{Eqn_StokesBVP}): 
finding $ \mathbf{u} \in H^1(\Omega)^d$ and $ p \in L^2(\Omega) $
such that $ \mathbf{u}|_{\partial\Omega} = \mathbf{g} $ (in the weak sense) and
\begin{equation}
\begin{cases}
  \mu (\nabla\mathbf{u}, \nabla\mathbf{v}) - (p, \nabla\cdot\mathbf{v})
= (\mathbf{f}, \mathbf{v}),
   \quad \forall \mathbf{v} \in H^1_0(\Omega)^d,
  \\ 
  -(\nabla\cdot\mathbf{u}, q)
  =  0,
 \quad \forall q \in L^2(\Omega) .
\end{cases}
\label{VarForm}
\end{equation}
Let $\mathcal{T}_h = \{K\}$ be a connected quasi-uniform simplicial mesh on $\Omega$.
A mesh is said to be connected if any pair of its elements is connected at least by a chain of elements that
share an interior facet with each other.
Define the discrete weak function spaces as
\begin{align}
     \displaystyle
    \mathbf{V}_h
    & = \{ \mathbf{u}_h = \{ \mathbf{u}_h^\circ, \mathbf{u}_h^\partial \}: \;
      \mathbf{u}_h^\circ|_{K} \in P_0(K)^d, \;
      \mathbf{u}_h^\partial|_e \in P_0(e)^d, \;
      \forall K \in \mathcal{T}_h, \; e \in \partial K \},
    \\ 
    \displaystyle
    W_h &= \{ p_h \in L^2(\Omega): \; p_h|_{K} \in P_0(K), \; \forall K \in \mathcal{T}_h\},
\end{align}
where $P_0(K)$ and $P_0(e)$ denote the sets of constant polynomials defined on element $K$ and facet $e$, respectively.
Note that $\mathbf{u}_h \in \mathbf{V}_h$ is approximated on both interiors and facets of mesh elements
while $p_h \in W_h$ is approximated on element interiors only.

Denote the lowest-order Raviart-Thomas space by $RT_0(K)$, i.e.,
\begin{align*}
    RT_0(K) = (P_0(K))^d + \V{x} \, P_0(K).
\end{align*}
Then, for a scalar function or a component of a vector-valued function, $u_h = (u_h^{\circ},u_h^{\partial})$,
the discrete weak gradient operator $\nabla_w: W_h \rightarrow RT_0(\mathcal{T}_h)$ is defined as
\begin{equation}
\label{weak-grad-1}
  (\nabla_w u_h, \mathbf{w})_K
  = (u^\partial_h, \mathbf{w} \cdot \mathbf{n})_{\partial K}
  - ( u^\circ_h , \nabla \cdot \mathbf{w})_K,
  \quad \forall \mathbf{w} \in RT_0(K),\quad \forall K \in \mathcal{T}_h ,
\end{equation}
where $\mathbf{n}$ is the unit outward normal to $\partial K$ and $(\cdot, \cdot)_K$
and $( \cdot, \cdot )_{\partial K}$ are the $L^2$ inner product on $K$ and $\partial K$, respectively.
For a vector-valued function $\mathbf{u}_h$, $\nabla_w \mathbf{u}_h$ is viewed as a matrix with each row representing
the weak gradient of a component.
By choosing $\V{w}$ in (\ref{weak-grad-1}) properly and using the fact that $\nabla_w u_h \in RT_0(K)$, we can obtain (e.g., see \cite{HuangWang_CiCP_2015})
\begin{align}
   & \nabla_w \varphi_K^{\circ} = - C_{K} (\V{x}-\V{x}_{K}) ,
   \label{grad_int}
   \\
   & \nabla_w \varphi_{K,i}^{\partial} = \frac{C_{K}}{d+1} (\V{x} -\V{x}_{K})
+ \frac{|e_{K,i}|}{|K|} \V{n}_{K,i}, \quad i = 1, ...,d+1,
\label{grad_face}
\end{align}
where $\varphi_K^{\circ}$ and $\varphi_{K,i}^{\partial}$ denote the basis functions of $P_0(K)$
and $P_0(e_{K,i})$, respectively, $e_{K,i}$ denotes the $i$-th facet of $K$,
$\V{n}_{K,i}$ is the unit outward normal to $e_{K,i}$,
\begin{align*}
C_{K} = \frac{d\; |K|}{\| \V{x} - \V{x}_{K} \|_{K}^2} ,
\quad
\V{x}_{K} = \frac{1}{d+1} \sum_{i=1}^{d+1} \V{x}_{K,i} ,
\end{align*}
and $\V{x}_{K,i}$, $i = 1, ..., d+1$ denote the vertices of $K$.

The discrete weak divergence operator $\nabla_w \cdot: \mathbf{V}_h \to \mathcal{P}_0(\mathcal{T}_h)$ needs to be defined
separately as 
\begin{equation}
   (\nabla_w \cdot \mathbf{u}, w )_{K}
  = ( \mathbf{u}^\partial , w \mathbf{n})_{ e }
  - ( \mathbf{u}^\circ , \nabla w)_{K},
  \quad
  \forall w \in P_0(K) .
  \label{wk_div1}
\end{equation}
Note that $\nabla_w \cdot \V{u}|_K \in P_0(K)$. By taking $w = 1$, we have
\begin{equation}
(\nabla_w \cdot \V{u}, 1)_{K} = \sum_{i=1}^{d+1} |e_{K,i}| \langle\V{u}\rangle_{e_{K,i}}^T \V{n}_{K,i} ,
\label{wk_div2}
\end{equation}
where $\langle\V{u}\rangle_{e_{K,i}}$ denotes the average of $\V{u}$ on facet $e_{K,i}$
and $|e_{K,i}|$ is the $(d-1)$-dimensional measure of $e_{K,i}$.

Having defined the discrete weak spaces, gradient, and divergence, we can now define the WG approximation
of (\ref{VarForm}): finding $ \mathbf{u}_h \in \mathbf{V}_h $ and $ p_h \in W_h $
such that $ \mathbf{u}_h^\partial|_{\partial \Omega} = Q_h^{\partial}\mathbf{g} $ and
\begin{equation}
\begin{cases}
    \displaystyle
    \mu \sum_{K\in\mathcal{T}_h} (\nabla_w \mathbf{u}_h,\nabla_w \mathbf{v})_K 
    -\sum_{K \in \mathcal{T}_h}(p_h^{\circ}, \nabla_w\cdot\mathbf{v})_K
    =  \sum_{K \in \mathcal{T}_h} (\mathbf{f}, \mathbf{\Lambda}_h\mathbf{v})_K,
      \quad \forall \mathbf{v} \in \mathbf{V}_h^0,
    \\ 
    \displaystyle
    -\sum_{K \in \mathcal{T}_h}(\nabla_w\cdot\mathbf{u}_h,q^{\circ})_K
    =  0,
   \quad \forall q \in W_h,
\end{cases}
\label{scheme}
\end{equation}
where $Q_h^{\partial}$ is a $L^2$-projection operator onto $\V{V}_h$ restricted on each facet
and the lifting operator $\mathbf{\Lambda}_h: \V{V}_h \to RT_0(\mathcal{T}_h)$ is defined \cite{Mu.2020,WANG202290} as
\begin{equation}
    \displaystyle
    ( (\mathbf{\Lambda}_h\mathbf{v}) \cdot \mathbf{n}, w )_{e}
      = ( \mathbf{v}^{\partial}\cdot\mathbf{n}, w )_{e},
        \quad \forall w \in P_0(e),\; \forall \V{v} \in \V{V}_h, \; \forall e \subset \partial K .
\label{Eqn_DefLambdah}
\end{equation}
Notice that $\mathbf{\Lambda}_h\mathbf{v}$ depends on $\mathbf{v}^{\partial}$
but not on $\mathbf{v}^{\circ}$. 
The following lemma shows the optimal-order convergence of scheme \eqref{scheme}.

\begin{lem}
\label{pro:err1}
Let $ \mathbf{u} \in H^{2}(\Omega)^d $ and $p \in H^1(\Omega)$ be the exact solutions for the Stokes problem (\ref{VarForm}) and let $\mathbf{u}_h\in \mathbf{V}_h$ and $p_h \in W_h$ be the numerical solutions of the scheme \eqref{scheme}.
Assume that $ \mathbf{f} \in L^2(\Omega)^d $. 
Then, there hold
\begin{align}
& \| p - p_h \| \le C h \|\V{f}\|,
\label{err-p}
\\
&   \| \nabla \mathbf{u} - \nabla_w \mathbf{u}_h \|
  \leq C h \|\mathbf{u}\|_2,
\label{err-du}
  \\
  & \| \V{u} - \V{u}_h \| = \| \V{u} - \V{u}_h^{\circ} \| \leq C h \|\mathbf{u}\|_2,
\label{err-u}
\\
&   \| Q_h^\circ\mathbf{u} - \mathbf{u}_h^\circ\|
  \leq C h^{2} \|\mathbf{u}\|_2,
\label{err-u0}
\end{align}
where $\| \cdot \| = \| \cdot \|_{L^2(\Omega)}$, $\|\cdot\|_2 = \|\cdot\|_{H^2(\Omega)}$, 
$C$ is a constant independent of $ h $ and $ \mu $, and $Q_h^\circ$ is a $L^2$-projection operator
for element interiors satisfying $Q_h^\circ\mathbf{u}|_K = \langle \V{u}\rangle_K, \;  \forall K \in \mathcal{T}_h$.
\end{lem}

\begin{proof}
The proof of these results can be found in \cite[Theorem 4.5]{MuYeZhang_SISC_2021} and \cite[Theorem 3]{WANG202290}.
\end{proof}

We would like to cast (\ref{scheme}) in a matrix-vector form. To this end,
for any element $K$ we denote
the WG approximations of $\V{u}$ on the interior and facets of $K$
by $\V{u}_{h,K}^{\circ}$ and $\V{u}_{h,K,i}^{\partial} \; (i = 1,...,d+1)$, respectively.
With this, we can express $\V{u}_h$ as
\begin{align*}
\displaystyle
\V{u}_{h}(\V{x}) & = \V{u}_{h,K}^{\circ} \varphi_K^{\circ}(\V{x}) + \sum_{i = 1}^{d+1} \V{u}_{h,K,i}^{\partial}
\varphi_{K,i}^{\partial}(\V{x})
= \V{u}_{h,K}^{\circ} \varphi_K^{\circ}(\V{x})
\\
& \qquad \qquad + \sum_{\substack{i = 1\\ e_{K,i} \notin \partial \Omega}}^{d+1} \V{u}_{h,K,i}^{\partial}
\varphi_{K,i}^{\partial}(\V{x})
+ \sum_{\substack{i = 1\\ e_{K,i} \in \partial \Omega}}^{d+1} \V{u}_{h,K,i}^{\partial}
\varphi_{K,i}^{\partial}(\V{x}),
\quad \forall \mathbf{x} \in K ,\; \forall K \in \mathcal{T}_h .
\end{align*}
Then, we can write \eqref{scheme} into a matrix-vector form as
\begin{equation}
    \begin{bmatrix}
        \mu A & -(B^{\circ})^T \\
       -B^{\circ} & \mathbf{0}
    \end{bmatrix}
    \begin{bmatrix}
        \mathbf{u}_h \\
        \mathbf{p}_h
    \end{bmatrix}
    =
    \begin{bmatrix}
        \mathbf{b}_1 \\
        \mathbf{b}_2
    \end{bmatrix},
    \label{scheme_matrix}
\end{equation}
where the matrices $A$ and $B^{\circ}$ and vectors $\V{b}_1$ and $\V{b}_2$ are defined as
\begin{align}
    & \mathbf{v}^T A \mathbf{u}_h  =  \sum_{K\in\mathcal{T}_h} (\nabla_w \mathbf{u}_h,\nabla_w \mathbf{v})_K
    = \sum_{K\in\mathcal{T}_h} (\V{u}_{h,K}^{\circ}\nabla_w \varphi_K^{\circ},  \nabla_w \mathbf{v})_K
    \label{A-1}  \\
   & 
   \displaystyle
    \qquad + \sum_{K\in\mathcal{T}_h} \sum^{d+1}_{\substack{i = 1 \\e_{K,i} \notin \partial \Omega}}(\V{u}_{h,K,i}^{\partial}\nabla_w \varphi_{K,i}^{\partial},\nabla_w \mathbf{v})_K,
    \quad \forall \mathbf{u}_h, \mathbf{v} \in \mathbf{V}_h^0,
    \notag \\
    & \mathbf{q}^T B^{\circ} \mathbf{u}_h  =  \sum_{K \in \mathcal{T}_h}  (\nabla_{w}\cdot\mathbf{u}_h,q^{\circ})_K
    \label{B-1} \\
& \qquad = \sum_{K\in\mathcal{T}_h} \sum^{d+1}_{\substack{i = 1 \\e_{K,i} \notin \partial \Omega}}
    |e_{K,i}| q_{K}^{\circ}(\V{u}_{h,K,i}^{\partial})^T \V{n}_{K,i},
    \quad \forall \mathbf{u}_h \in \mathbf{V}_h^0, \quad \forall q \in W_h
    \notag 
    \\
& \mathbf{v}^T \V{b}_1 = \sum_{K \in \mathcal{T}_h} (\mathbf{f}, \mathbf{\Lambda}_h\mathbf{v})_K
    - \mu \sum_{K\in\mathcal{T}_h} \sum_{\substack{i = 1\\ e_{K,i} \in \partial \Omega}}^{d+1}( (Q_{h}^{\partial}\V{g}) \nabla_w \varphi_{K,i}^{\partial},\nabla_w \mathbf{v})_K, \; \forall \mathbf{v} \in \mathbf{V}_h^0,
    \label{b1-1}
    \\
& \mathbf{q}^T \V{b}_2 = \sum_{K\in\mathcal{T}_h} \sum^{d+1}_{\substack{i = 1 \\e_{K,i} \in \partial \Omega}}
    |e_{K,i}| q_{K}^{\circ} (Q_{h}^{\partial}\V{g}|_{e_{K,i}} )^T \V{n}_{K,i},
    \quad \forall q \in W_h .
    \label{b2-1}
\end{align}
In the above equations, we have used $\V{v}$ (or $\V{v}_h$) interchangeably for any WG approximation of $\V{v}_h$ in $\V{V}_h$
and the vector formed by its values $(\V{v}_{h,K}^{\circ}, \V{v}_{h,K,i}^{\partial})$ for $i = 1, ..., d+1$ and $K \in \mathcal{T}_h$,
excluding those on $\partial \Omega$.
Similarly, for any $q_h \in W_h$, $\V{q}$ (or $\V{q}_h$) is used to denote the vector formed by $q_{h,K}$ for all $K \in \mathcal{T}_h$.

Notice that (\ref{scheme_matrix}) is a saddle point system.
We are interested in its efficient iterative solution using block Schur complement preconditioning.
The following lemma shows that (\ref{scheme_matrix}) is singular and the pressure solution is not unique.

\begin{lem}
\label{lem:B0-1}
The null space of $(B^{\circ})^T$ is given by
\[
\text{Null}((B^{\circ})^T) = \{ p_h \in W_h: \; p_{h,K} = C, \; \forall K \in \mathcal{T}_h, \; \text{ C is a constant} \} .
\]
\end{lem}

\begin{proof}
For any $p_h \in \text{Null}((B^{\circ})^T)$, from (\ref{B-1}) we have
\[
\V{v}^T (B^{\circ})^T \V{p}_h = \V{p}_h^T B^{\circ} \V{v}
= \sum_{K\in\mathcal{T}_h} \sum^{d+1}_{\substack{i = 1 \\e_{K,i} \notin \partial \Omega}}
    |e_{K,i}| p_{h,K}^{\circ}(\V{v}_{K,i}^{\partial})^T \V{n}_{K,i},
    \quad \forall \mathbf{v} \in \mathbf{V}_h^0 .
\]
Let $e$ be an arbitrary interior facet and the two elements sharing it be $K$ and $\tilde{K}$.
Taking $\V{v}|_e = \V{n}_e$ and $\V{v} = \V{0}$ elsewhere in the above equation, we obtain
\[
    (p_{h,K}^{\circ} - p_{h,\tilde{K}}^{\circ}) |e| = 0,
\]
which implies $p_{h,K}^{\circ} = p_{h,\tilde{K}}^{\circ}$. From the arbitrariness of $e$ and the connection assumption of the mesh, we know
that $p_h$ is constant on all elements.
\end{proof}

The above lemma implies that $(B^{\circ})^T$ is one-rank deficient. As a consequence, the linear system
(\ref{scheme_matrix}) is singular and the pressure solution is not unique.
Here, we follow a strategy commonly used to avoid the singularity and make the pressure solution unique
by specifying the value of the pressure at a specific location.
To this end, we modify the zero (2,2)-block of \eqref{scheme_matrix} into
\begin{equation}
    \begin{bmatrix}
        \mu A & -(B^{\circ})^T \\
       -B^{\circ} & -D
    \end{bmatrix}
    \begin{bmatrix}
        \mathbf{u}_h \\
        \mathbf{p}_h
    \end{bmatrix}
    =
    \begin{bmatrix}
        \mathbf{b}_1 \\
        \mathbf{b}_2
    \end{bmatrix},
    \label{regularized_scheme}
\end{equation}
where $D = \text{diag}(d_{11}, 0, ..., 0)$ and $d_{11}$ is a positive number whose choice will be discussed later.
This modification is equivalent to adding $d_{11} p_{h,K_1}^{\circ} = 0$ to the first equation of the second-block
of scheme (\ref{scheme_matrix}), which can be viewed as specifying the value zero of the pressure at the barycenter
of the first mesh element.
Moreover, this regularization does not change the solution of the system: the solution of \eqref{regularized_scheme}
is a solution of  (\ref{scheme_matrix}). Furthermore, the regularization maintains the symmetry and sparseness
of the coefficient matrix.

The other regularization strategies include
zero-mean condition enforcement and projection methods (e.g., see \cite{GUERMOND20066011,GWYNLLYW20061027})
and global regularization methods (such as the one of \cite{GWYNLLYW20061027} where
a scalar multiple of the mass matrix of pressure is added to the zero block of the system matrix).

In the following, we provide an analytical proof of the nonsingularity of \eqref{regularized_scheme} and establish bounds for the Schur complement $S$.

By rescaling the unknown variables, we rewrite \eqref{regularized_scheme} into
\begin{equation}
    \begin{bmatrix}
        A & -(B^{\circ})^T \\
       -B^{\circ} & -\mu D
    \end{bmatrix}
    \begin{bmatrix}
        \mu \mathbf{u}_h \\
        \mathbf{p}_h
    \end{bmatrix}
    =
    \begin{bmatrix}
        \mathbf{b}_1 \\
        \mu \mathbf{b}_2
    \end{bmatrix},
    \quad \mathcal{A}   = \begin{bmatrix}
        A & -(B^{\circ})^T \\
       -B^{\circ} & -\mu D
    \end{bmatrix} .
    \label{scheme_matrix_2}
\end{equation}

\begin{lem}
\label{lem:S-2}
The Schur complement $ S = \mu D + B^{\circ} A^{-1} (B^{\circ})^T$ for (\ref{scheme_matrix_2}) is symmetric and positive
definite. Moreover, there holds
\begin{align}
    \mu D \le S \le \left (d + \mu \frac{d_{11}}{|K_1|}\right ) M_p^{\circ},
\label{lem:S-2-1}
\end{align}
where $|K_1|$ is the volume of the first element.
\end{lem}
\begin{proof}
It is obvious that $S$ is symmetric and positive semi-definite. We just need to show that $S$ is nonsingular
for its positive definiteness. Assume that $S$ is singular. Then, there exists a non-zero vector $\V{p}$ such that
$\V{p}^T S \V{p} = 0$. This implies that $\mu p_1^2 d_{11} + \V{p}^T B^{\circ} A^{-1} (B^{\circ})^T \V{p} = 0$,
where $p_1$ is the first component of $\V{p}$. Thus,
we have $p_1 = 0$ and $\V{p}^T B^{\circ} A^{-1} (B^{\circ})^T \V{p} = 0$. In other words, $\V{p} \in \text{Null}((B^{\circ})^T)$
and $p_1 = 0$. By Lemma~\ref{lem:B0-1}, this implies $\V{p} = \V{0}$,
which is in contradiction with the fact that $\V{p}$ is a non-zero
vector. Thus, $S$ should be nonsingular and therefore, symmetric and positive definite.

From the definitions of operator $B^{\circ}$ in \eqref{B-1} and the mass matrix  and the fact
$\nabla_w \cdot \V{u} \in P_0(\mathcal{T}_h)$, it is not difficult to get
\begin{align*}
\V{u}^{T}  (B^{\circ})^T (M_p^{\circ})^{-1} B^{\circ} \V{u} = \sum_{K \in \mathcal{T}_h}(\nabla_w \cdot \V{u}, \nabla_w \cdot \V{u})_K .
\end{align*}
Moreover, it can be verified directly that
\begin{align*}
   \sum_{K \in \mathcal{T}_h}(\nabla_w \cdot \mathbf{u_h},\nabla_w \cdot \mathbf{u_h})_K \le
d \sum_{K \in \mathcal{T}_h}(\nabla_w \mathbf{u}_h,\nabla_w \mathbf{u}_h )_K .
\end{align*}
Then, since both $A$ (cf. \eqref{scheme_matrix}) and $M_p^{\circ} = \text{diag}(|K_1|, ..., |K_N|)$
are symmetric and positive definite, we have
\begin{align}
    \sup_{\mathbf{p} \neq 0} \frac{\V{p}^{T}( \mu D +  B^{\circ}{A}^{-1}(B^{\circ})^T) \V{p}}{\V{p}^T M_p^{\circ} \V{p}} 
    & = \sup_{\mathbf{p} \neq 0}  \frac{\V{p}^{T} (M_p^{\circ})^{-\frac12} ( \mu D +  B^{\circ}{A}^{-1}(B^{\circ})^T)
    (M_p^{\circ})^{-\frac12} \V{p}}{\V{p}^T \V{p}}
    \notag
    \\
    & \le \mu \frac{d_{11}}{|K_1|} + \sup_{\mathbf{u} \neq 0} \frac{\V{u}^{T}  (B^{\circ})^T (M_p^{\circ})^{-1} B^{\circ}
    \V{u}}{\V{u}^T A \V{u}} ,
    \notag \\
    & = \mu \frac{d_{11}}{|K_1|} + \sup_{\mathbf{u} \neq 0} \frac{\sum_{K \in \mathcal{T}_h}(\nabla_w \cdot \mathbf{u_h},\nabla_w \cdot \mathbf{u_h})_K}{\sum_{K \in \mathcal{T}_h}(\nabla_w \mathbf{u}_h,\nabla_w \mathbf{u}_h )_K}
    \notag \\
    & \le \mu \frac{d_{11}}{|K_1|} + d,
    \label{AB-1}
\end{align}
which implies $S \le (\mu \frac{d_{11}}{|K_1|} + d )M_p^{\circ}$.
\end{proof}

The focus of this work is on the efficient iterative solution of the saddle point system (\ref{scheme_matrix_2}) with block
Schur complement preconditioning. The solution of general saddle point systems has been extensively studied
and remains a topic of active research; e.g., see review articles \cite{BenziGolubLiesen-2005,Benzi2008}
and more recent works \cite{Ainsworth_SINUM_2022,Bacuta-2019,Rhebergen_SISC_2022}.
Most systems that have been studied are either singular systems with a single eigenvalue exactly equal to zero
and other eigenvalues away from zero or nonsingular systems with eigenvalues away from zero.
On the other hand, for the current system (\ref{scheme_matrix_2}), \eqref{lem:S-2-1} does not imply
the spectral equivalent between $\hat{S}$ and $M_p^{\circ}$. As a matter of fact and 
will be seen in later sections, $S$ has a small eigenvalue that approaches zero as $\mu \to 0$.
This almost singular but nonsingular feature makes (\ref{scheme_matrix_2}) distinctive from saddle point systems that have
been well studied and poses challenges in developing effective preconditioners and
convergence analysis of its iterative solution. 

To prepare the discussion on block Schur complement preconditioning for (\ref{scheme_matrix_2}),
we provide a brief discussion on (inexact) block diagonal and block triangular Schur complement preconditioners
for general saddle-point problems in Appendix~\ref{appendix:gmres convergence}.
Particularly, we establish estimates for the residual of GMRES for preconditioned systems
using block upper and lower triangular preconditioners in Lemmas~\ref{lem:gmres-conv-lower} and \ref{lem:gmres-conv-upper}.

In next two sections, we consider block diagonal/triangular Schur complement preconditioners for (\ref{scheme_matrix_2})
and the convergence of MINRES/GMRES accordingly.

\section{Convergence of MINRES with block diagonal Schur complement preconditioning}
\label{sec::regularization-diagonal}
In this section we consider the block diagonal Schur complement preconditioning for the regularized system \eqref{scheme_matrix_2}.

From Lemma~\ref{lem:S-2}, we take $\hat{S} = M_p^{\circ}$ as an approximation to the Schur complement and
use the block diagonal Schur complement preconditioner 
\begin{align}
    \mathcal{P}_d = \begin{bmatrix}
        A & 0 \\
        0 & M_p^{\circ}
    \end{bmatrix}.
    \label{PrecondPd}
\end{align}
Since $\mathcal{P}_d$ is SPD, the preconditioned system $\mathcal{P}_d^{-1} \mathcal{A}$ is similar to
$\mathcal{P}_d^{-\frac{1}{2}} \mathcal{A} \mathcal{P}_d^{-\frac{1}{2}}$ which can be expressed as
\begin{align}
    \mathcal{P}_d^{-\frac{1}{2}} \mathcal{A} \mathcal{P}_d^{-\frac{1}{2}}&= 
    \begin{bmatrix}
        A^{\frac{1}{2}} & 0 \\
        0 &(M_p^{\circ})^{\frac{1}{2}}
    \end{bmatrix}^{-1}
    \begin{bmatrix}
        A & -(B^{\circ})^T \\
       -B^{\circ} & -\mu D
    \end{bmatrix}
    \begin{bmatrix}
        A^{\frac{1}{2}} & 0 \\
        0 &(M_p^{\circ})^{\frac{1}{2}}
    \end{bmatrix}^{-1} \notag
    \\ 
    & = \begin{bmatrix}
        \mathcal{I} & -A^{-\frac{1}{2}} (B^{\circ})^T(M_p^{\circ})^{-\frac{1}{2}} \\
        -(M_p^{\circ})^{-\frac{1}{2}} B^{\circ}A^{-\frac{1}{2}} & 0
    \end{bmatrix}
    +
    \begin{bmatrix}
        0 & 0 \\
        0 & -\mu (M_p^{\circ})^{-\frac{1}{2}} D (M_p^{\circ})^{-\frac{1}{2}}
    \end{bmatrix}.
    \label{diag-precond-system}
\end{align}
This means that we can use MINRES for the iterative solution of the preconditioned system and analyze
its convergence with spectral analysis.
Since $\mu$ is a small value, the matrix with block $-\mu (M_p^{\circ})^{-\frac{1}{2}} D (M_p^{\circ})^{-\frac{1}{2}}$ can be considered a perturbation of the first matrix in \eqref{diag-precond-system}.
The 
we use the Bauer-Fike theorem  to obtain the eigenvalues of the preconditioned system.

\begin{lem}
\label{lem:eigen_bound_diag}
The eigenvalues of $ \mathcal{P}_d^{-1} \mathcal{A} $ are bounded by 
\begin{align}
    \Bigg[ \frac{1-\sqrt{1+4 d}}{2} - \mu \frac{d_{11}}{|K_1|}, \,
    &\frac{1-\sqrt{1+4 \beta^2}}{2} + \mu \frac{d_{11}}{|K_1|} \Bigg] 
    \cup \Bigg\{  - \mu \frac{d_{11}}{|\Omega|} + \mathcal{O}(\mu^2) \Bigg\} \nonumber \\
    &\cup \Bigg[ \frac{1+\sqrt{1+4 \beta^2}}{2} - \mu \frac{d_{11}}{|K_1|}, \,
    \frac{1+\sqrt{1+4 d}}{2} + \mu \frac{d_{11}}{|K_1|} \Bigg] ,
    \label{eigen_bound_diag-1}
\end{align}
provided that
\begin{align}
\mu \frac{d_{11}}{|K_1|} \ll \frac{\sqrt{1+4 \beta^2}-1}{2} .
\label{eigen_bound_diag-2}
\end{align}
\end{lem}

\begin{proof}
First, we consider the eigenvalue problem of the unperturbed preconditioned system (i.e., $\mu = 0$),
\begin{align}
    \begin{bmatrix}
        A & -(B^{\circ})^T \\
       -B^{\circ} &  0 
    \end{bmatrix}
    \begin{bmatrix}
        \V{u} \\
        \V{p}
    \end{bmatrix} = \lambda
   \begin{bmatrix}
        A & 0 \\
        0 & M_p^{\circ}
    \end{bmatrix} 
        \begin{bmatrix}
        \V{u} \\
        \V{p}
    \end{bmatrix}
    \label{mu0system}
\end{align}
It is readily seen that $\lambda = 1$ is not an eigenvalue to the problem. 
Then the eigenvalues satisfy
\begin{align}
    \lambda (\lambda - 1) M_p^{\circ} \V{p} = B^{\circ}A^{-1}(B^{\circ})^T\V{p} .
    \notag
\end{align}
Denoting the eigenvalues of $(M_p^{\circ})^{-1}B^{\circ} A^{-1} (B^{\circ})^T $ by $\gamma_1 = 0 < \gamma_2 \le \cdots \le \gamma_N$.
From (\ref{AB-1}),
we have $\gamma_N \le d$. Moreover,
the smallest positive eigenvalue of $(M_p^{\circ})^{-1/2} B^{\circ} A^{-1} (B^{\circ})^T$ is equal to
that of $A^{-1/2} (B^{\circ})^T (M_p^{\circ})^{-1} B^{\circ} A^{-1/2}$, with the latter being square
of the inf-sup constant $\beta$. Thus, $\gamma_2 = \beta^2$.
The inf-sup condition for the WG approximation of the Stokes problem has been proved
in \cite{WANG202290}.
Therefore, eigenvalues of \eqref{mu0system} satisfy
\begin{align}
    \lambda^2 - \lambda - \gamma_i = 0, \quad i = 2, ..., N,
    \notag
\end{align}
or
\begin{align}
    \lambda = \frac{1 \pm \sqrt{1+4 \gamma_i}}{2},
    \notag
\end{align}
where $\gamma_1 = 0$ and $\gamma_i \in [\beta^2,d]$ for $i = 2, ..., N$.
From this, we obtain the bounds for the eigenvalues of \eqref{mu0system} as
\begin{align}
    \Bigg[ \frac{1-\sqrt{1+4 d}}{2}, \frac{1-\sqrt{1+4 \beta^2}}{2}\Bigg]\cup \{ 0 \}\cup \Bigg [ \frac{1+\sqrt{1+4 \beta^2}}{2}, \frac{1+\sqrt{1+4 d}}{2}\Bigg] .
    \label{eig_1}
\end{align}

Now we turn back to \eqref{diag-precond-system}.
Notice that
\[
\| \mu (M_p^{\circ})^{-\frac{1}{2}} D (M_p^{\circ})^{-\frac{1}{2}} \| \le \mu \frac{d_{11}}{|K_1|} .
\]
By the Bauer-Fike theorem (e.g., see \cite[Corollary 6.5.8]{Watkins-2010}), we know that if (\ref{eigen_bound_diag-2}) is satisfied,
the eigenvalues of \eqref{diag-precond-system} are bounded by
\begin{align}
    \Bigg[ \frac{1-\sqrt{1+4 d}}{2} - \mu \frac{d_{11}}{|K_1|}, \,
    &\frac{1-\sqrt{1+4 \beta^2}}{2} + \mu \frac{d_{11}}{|K_1|} \Bigg] 
    \cup \{ \lambda_1(\mu) \} \nonumber \\
    &\cup \Bigg[ \frac{1+\sqrt{1+4 \beta^2}}{2} - \mu \frac{d_{11}}{|K_1|}, \,
    \frac{1+\sqrt{1+4 d}}{2} + \mu \frac{d_{11}}{|K_1|} \Bigg],
    \label{eig_2}
\end{align}
which gives (\ref{eigen_bound_diag-1}) except for the eigenvalue $\lambda_1(\mu)$,
a perturbation of the zero eigenvalue. 

Now, we estimate $\lambda_1(\mu)$. Denote $ \mathcal{B}(\mu)  = \mathcal{P}_d^{-1/2} \mathcal{A} \mathcal{P}_d^{-1/2}$.
Consider the eigenvalue problem
\begin{align*}
    \mathcal{B}(\mu) 
\begin{bmatrix}
    \V{u}(\mu) \\ \V{p}(\mu)
\end{bmatrix} = 
\lambda_1(\mu) \begin{bmatrix}
    \V{u}(\mu) \\ \V{p}(\mu)
\end{bmatrix}.
\end{align*}
Differentiating the above equation with respect to $\mu$, we have
\begin{align}
    \frac{\partial \mathcal{B}}{\partial \mu} (\mu) \begin{bmatrix}
    \V{u}(\mu) \\ \V{p}(\mu)
\end{bmatrix}  +  \mathcal{B}(\mu) 
\begin{bmatrix}
    \frac{\partial \V{u} }{\partial \mu}(\mu) \\ \frac{\partial \V{p}}{\partial \mu} (\mu)
\end{bmatrix}  
 = 
    \frac{\partial \lambda_1}{\partial \mu}(\mu) \begin{bmatrix}
    \V{u}(\mu) \\ \V{p}(\mu) 
\end{bmatrix} 
+ \lambda_1(\mu)
\begin{bmatrix}
    \frac{\partial \V{u} }{\partial \mu}(\mu) \\ \frac{\partial \V{p}}{\partial \mu} (\mu)
\end{bmatrix}    .
\label{eigenvalue_sys_1}
\end{align}
Notice that $\lambda_1(0) = 0$, $\V{u}(0) = 0$, and $\V{p}(0) = \V{v}_1$, where
\begin{align}
    \V{v}_1 = \frac{1}{\sqrt{|\Omega|}} \left (
    |K_1|^{\frac 1 2}, ..., |K_N|^{\frac 1 2} \right )^T .
    \label{v1-eigen}
\end{align}
Taking $\mu = 0$ in (\ref{eigenvalue_sys_1}) and multiplying with $[0, \V{v}_1^T]$ from left, we get
\begin{align*}
   &\begin{bmatrix}
    0 & \V{v}_1^T
\end{bmatrix}  \frac{\partial \mathcal{B}}{\partial \mu} (0) \begin{bmatrix}
    0 \\ \V{v}_1 
\end{bmatrix}  +  \begin{bmatrix}
    0 & \V{v}_1^T
\end{bmatrix}  \mathcal{B}(0)  
\begin{bmatrix}
    \frac{\partial \V{u} }{\partial \mu} (0) \\ \frac{\partial \V{p} }{\partial \mu} (0)
\end{bmatrix}\begin{bmatrix}
    0 \\ \V{v}_1 
\end{bmatrix}  \notag \\
 & = 
     \begin{bmatrix}
    0 & \V{v}_1^T
\end{bmatrix}  \frac{\partial \lambda_1}{\partial \mu}(0) \begin{bmatrix}
    0 \\ \V{v}_1 
\end{bmatrix} 
+ \lambda_1(0) \begin{bmatrix}
    0 & \V{v}_1^T
\end{bmatrix}
\begin{bmatrix}
    \frac{\partial \V{u}}{\partial \mu}(0) \\ \frac{\partial \V{p}}{\partial \mu}(0)
\end{bmatrix}  \begin{bmatrix}
    0 \\ \V{v}_1 
\end{bmatrix} .
\end{align*}
From this and the fact that
\[
\frac{\partial \mathcal{B}}{\partial \mu} = 
\begin{bmatrix}
    0 & 0 \\
    0 & - (M_p^{\circ})^{-\frac{1}{2}} D (M_p^{\circ})^{-\frac{1}{2}}
\end{bmatrix},
\]
we obtain
\begin{align*}
   \frac{\partial \lambda_1}{\partial \mu} ( 0)\V{v}_1^T \V{v}_1 = -\V{v}_1^T (M_p^{\circ})^{-\frac{1}{2}} D (M_p^{\circ})^{-\frac{1}{2}} \V{v}_1 ,
\end{align*}
which, with the expressions of $\V{v}_1$, $D$, and $M_p^{\circ}$, yields
\begin{align*}
    \frac{\partial \lambda_1}{\partial \mu}( 0) = -\frac{d_{11}}{|\Omega|} .
\end{align*}
This gives the estimate of $\lambda_1(\mu)$ in (\ref{eigen_bound_diag-1}).
\end{proof}

\begin{pro}
    \label{MINRES_conv}
Assume that (\ref{eigen_bound_diag-2}) is satisfied.
Then the residual of MINRES applied to the preconditioned system $\mathcal{P}_{d}^{-1/2} \mathcal{A}\mathcal{P}_{d}^{-1/2} $ is bounded by
\begin{align}
    \frac{\| \V{r}_{2k+1}\| }{\| \V{r}_0 \|} \; {\stackrel{<}{\sim}} \; \frac{2|\Omega| (d+\mu \frac{d_{11}}{|K_1|} + \mu \frac{d_{11}}{|\Omega|})}{\mu d_{11}}
     \Bigg( \frac{\sqrt{d} - \beta}{\sqrt{d} + \beta}\Bigg)^k .
     \label{MINRES_res}
\end{align}
\end{pro}

\begin{proof}
It is known \cite{MINRES-1975} that the residual of MINRES is given by
\[
\| \V{r}_{2k+1}\| = \min\limits_{\substack{p \in \mathbb{P}_{2k+1}\\ p(0) = 1}} \| p (\mathcal{P}_{d}^{-1}\mathcal{A})  \V{r}_0 \|
\le \min\limits_{\substack{p \in \mathbb{P}_{2k+1}\\ p(0) = 1}} \| p (\mathcal{P}_{d}^{-1} \mathcal{A})\| \; \| \V{r}_0 \|,
\]
where $\mathbb{P}_{2k+1}$ is the set of polynomials of degree up to $2k+1$.
Denote the eigenvalues of $\mathcal{P}_{d}^{-1/2} \mathcal{A}\mathcal{P}_{d}^{-1/2} $ by $\lambda_i$, $ i = 1, ..., 2N-1$.
Also denote the intervals in (\ref{eigen_bound_diag-1}) (except for the eigenvalue near zero) by $[a_1,b_1]\cup [a_2,b_2]$.
From Theorem 6.13 of \cite{Elman-2014} (about the residual of MINRES) and
Lemma~\ref{lem:eigen_bound_diag}, we have
\begin{align}
    \frac{\| \V{r}_{2k+1}\| }{\| \V{r}_0 \|} &\le \min\limits_{\substack{p \in \mathbb{P}_{2k+1}\\ p(0) = 1}} \max\limits_{ i = 1,2,..,2N-1} | p (\lambda_i)| \notag
\\
 &\le
    \min\limits_{\substack{p \in \mathbb{P}_{2k}\\ p(0) = 1}} 
    \max_{i = 2, ..., 2 N-1} |\frac{(\lambda_i - \lambda_1)}{\lambda_1} p(\lambda_i) | \nonumber
    \\
   &\le \Big|\frac{d+ \mu \frac{d_{11}}{|K_1|} + \mu \frac{d_{11}}{|\Omega|}+ \mathcal{O}(\mu^2)}{\mu \frac{d_{11}}{|\Omega|} + \mathcal{O}(\mu^2)}\Big|
   \min\limits_{\substack{p \in \mathbb{P}_{2k}\\ p(0) = 1}} 
    \max_{\lambda \in  [a_1,b_1]\cup [a_2,b_2]} | p(\lambda) | \nonumber 
    \\
  &   \le 2\Big|\frac{d+ \mu \frac{d_{11}}{|K_1|} + \mu \frac{d_{11}}{|\Omega|}+ \mathcal{O}(\mu^2)}{\mu \frac{d_{11}}{|\Omega|} + \mathcal{O}(\mu^2)}\Big| \notag
  \\
  & \qquad \cdot
     \Bigg( \frac{\sqrt{(\mu \frac{d_{11}}{|K_1|})^2 + \mu\frac{d_{11}}{|K_1|} \sqrt{1+4d} + d } - \sqrt{(\mu \frac{d_{11}}{|K_1|})^2 - \mu\frac{d_{11}}{|K_1|} \sqrt{1+4\beta^2} + \beta^2 }}{\sqrt{(\mu \frac{d_{11}}{|K_1|})^2 + \mu\frac{d_{11}}{|K_1|} \sqrt{1+4d} + d } + \sqrt{(\mu \frac{d_{11}}{|K_1|})^2 - \mu\frac{d_{11}}{|K_1|} \sqrt{1+4\beta^2} + \beta^2 }} \Bigg)^k,
\notag
\end{align}
which leads to (\ref{MINRES_res}).
\end{proof}

It is worth mentioning that $\| \V{r}_{2k+2}\| \le \| \V{r}_{2k+1}\|$ holds
due to the minimization property of MINRES. 
Moreover, Proposition~\ref{MINRES_conv} indicates that the convergence factor of MINRES
applied to the regularized system  (\ref{scheme_matrix_2})
with the block diagonal preconditioner (\ref{PrecondPd}) is almost independent of $h$ and $\mu$.
On the other hand, the asymptotic error constant in \eqref{MINRES_res} appears to be related to $\mu$ and $h$. 
Since the number of MINRES iterations required to reach a prescribed level of the residual
is proportional to the logarithm of the asymptotic error constant, i.e., $\log (\mu d_{11})$.
This dependence is weak and acceptable in practical computation. 

Interestingly, there is no obvious optimal choice of $d_{11}$ to make $\log (\mu d_{11})$
to be independent of $h$ and $\mu$ while satisfying the condition (\ref{eigen_bound_diag-2})
(under which (\ref{MINRES_res}) is valid). 
Although this condition is needed only for the purpose
of theoretical analysis (i.e., it is not needed in the actual computation), we do not want to choose $d_{11}$ too large
so the eigenvalues of the preconditioned system spread out over places and get close to the origin.
One obvious choice is $d_{11} = |K_1|$. For this case, $\log (\mu d_{11}) = \log (\mu |K_1|)$
and (\ref{eigen_bound_diag-2}) becomes $\mu \ll \frac{1}{2} (\sqrt{1+4\beta^2}-1)$.
Another choice is $d_{11} = 1$, for which we have $\log (\mu d_{11}) = \log (\mu)$ 
but (\ref{eigen_bound_diag-2}) becomes $\mu \ll \frac{1}{2} (\sqrt{1+4\beta^2}-1) |K_1|$, which holds
for much smaller $\mu$ than in the previous choice.
Once again, for both cases, the number of MINRES iterations required to reach a prescribed level of the residual
depends only logarithmically on $\mu$ and/or $h$.

\section{Convergence of GMRES with block triangular Schur complement preconditioning}
\label{sec::regularization-triagle}
In this section we consider the block lower triangular Schur complement preconditioner,
\begin{align}
    \mathcal{P}_t = \begin{bmatrix}
        A & 0 \\
        -B^{\circ} & -M_p^{\circ}
    \end{bmatrix}.
    \label{PrecondP}
\end{align}
As discussed in Appendix~\ref{appendix:gmres convergence}, all four block triangular Schur complement
preconditioners in (\ref{SPP-4})
will perform similarly. To be specific, we only consider (\ref{PrecondP}) here.
It should be pointed out that with block triangular preconditioners, the corresponding preconditioned system
is no longer diagonalizable in general. This means that we need to use GMRES to solve the system.
Moreover, the spectral analysis is insufficient to determine the convergence of GMRES.
Fortunately, Lemma~\ref{lem:gmres-conv-lower} allows us to analyze the convergence of GMRES
for block lower triangular preconditioners through $\| A^{-1} (B^{\circ})^T \|$,
$\| \hat{S}^{-1}S\|$, and $\min_{p} \| p(\hat{S}^{-1}S) \|$.
This can be done similarly for block upper triangular preconditioners; cf. Lemma~\ref{lem:gmres-conv-upper}.

Recall that $ S = \mu D + B^{\circ} A^{-1} (B^{\circ})^T$ and $\hat{S} = M_p^{\circ}$.
For $\| A^{-1} (B^{\circ})^T \|$, using (\ref{AB-1}) and the fact that $A$ and $M_p^{\circ}$ are SPD,  we have
\begin{align}
    \| A^{-1} (B^{\circ})^T \|^2 
    &= \sup_{\V{p} \neq 0} \frac{\V{p}B^{\circ} A^{-1} A^{-1} (B^{\circ})^T \V{p}}{\V{p}^T  \V{p}} \nonumber
    \\
   & = \sup_{\V{p} \neq 0} \frac{\V{p} (M_p^{\circ})^{\frac{1}{2}} (M_p^{\circ})^{-\frac{1}{2}}B^{\circ} A^{-1}A^{-1} (B^{\circ})^T (M_p^{\circ})^{-\frac{1}{2}} (M_p^{\circ})^{\frac{1}{2}} \V{p}}{\V{p}^T \V{p}}
   \notag \\
   & \le \lambda_{\max} (A^{-1}) \lambda_{\max} (M_p^{\circ})
 \sup_{\V{u} \neq 0} \frac{\V{u}^T B^{\circ} (M_p^{\circ})^{-1} (B^{\circ})^T \V{u}}{\V{u}^T A\V{u}}
 \notag \\
 & \le \frac{d\, \lambda_{\max} (M_p^{\circ})}{\lambda_{\min} (A)} 
    \label{bound2}.
\end{align}
Moreover, from Lemma~\ref{lem:S-2} we have
\[
\| \hat{S}^{-1} S \| \le d+\mu \frac{d_{11}}{|K_1|} .
\]
Using the above results and Lemma~\ref{lem:gmres-conv-lower}, we obtain
\begin{align}
    \label{GMRES-residual-3}
    \frac{\| \V{r}_k\|}{\| \V{r}_0\|}
    \le \left (1+d+\mu \frac{d_{11}}{|K_1|} + \Big (\frac{d\, \lambda_{\max} (M_p^{\circ})}{\lambda_{\min} (A)}\Big )^\frac{1}{2}\right ) 
    \min\limits_{\substack{p \in \mathbb{P}_{k-1}\\ p(0) = 1}} \| p(\hat{S}^{-1}S) \| .
\end{align}

We estimate the eigenvalues of  $\hat{S}^{-1} S$ in the following.

\begin{lem}
\label{eigen_bound}
The eigenvalues of $\hat{S}^{-1} S$, $0 < \lambda_1 < \lambda_2 \le ... \le \lambda_N$, are bounded by
\begin{align}
         \mu\frac{d_{11}}{|\Omega|} \le \lambda_1 \le \mu \frac{d_{11}}{|K_1|}, \quad
\beta^2 - \mu\frac{d_{11}}{|K_1|} \le \lambda_i \le d +  \mu \frac{d_{11}}{|K_1|}, \quad i = 2, ..., N.
\label{eigen_bound-2}
\end{align}
\end{lem}

\begin{proof}
From previous section, we know eigenvalues of 
\begin{align}
(M_p^{\circ})^{-\frac{1}{2}}B^{\circ} A^{-1} (B^{\circ})^T(M_p^{\circ})^{-\frac{1}{2}}
\label{M-0}
\end{align}
are  $\gamma_1 = 0$ and $\gamma_i \in [\beta^2, d]$ for $i = 2, ..., N$, where $\beta$ is the inf-sup constant.
Moreover, $\hat{S}^{-1} S$ is similar to 
\begin{align}
\mu (M_p^{\circ})^{-\frac{1}{2}} D (M_p^{\circ})^{-\frac{1}{2}} +
(M_p^{\circ})^{-\frac{1}{2}}B^{\circ} A^{-1} (B^{\circ})^T(M_p^{\circ})^{-\frac{1}{2}} ,
\label{eigen_4}
\end{align}
where the first term can be viewed as a perturbation of the second term.
Using the Bauer-Fike theorem (e.g., see \cite[Corollary 6.5.8]{Watkins-2010})
and the fact that $S$ is symmetric and positive definite (cf. Lemma~\ref{lem:S-2}),
we obtain the bounds for the eigenvalues of $\hat{S}^{-1}S$ as
\begin{align*}
0 < \lambda_1 \le \mu\frac{d_{11}}{|K_1|}, \quad
\beta^2 - \mu\frac{d_{11}}{|K_1|} \le \lambda_i \le d +  \mu\frac{d_{11}}{|K_1|}, \quad i = 2, ..., N,
\end{align*}
which gives (\ref{eigen_bound-2}) except for the lower bound of $\lambda_1$.

For the lower bound of $\lambda_1$, recall that $\V{v}_1$ defned in (\ref{v1-eigen})
is an normalized eigenvector of the matrix (\ref{M-0})
associated with the zero eigenvalue. Denote the other normalized eigenvectors of the matrix 
by $\V{v}_i$, $i = 2,3,...,N$, i.e.,
\begin{align*}
    (M_p^{\circ})^{-\frac{1}{2}}B^{\circ} A^{-1} (B^{\circ})^T(M_p^{\circ})^{-\frac{1}{2}} \V{v}_i = \gamma_i \V{v}_i, \quad i = 2,...,N .
\end{align*}
Then, any vector $\V{v}\in \mathbb{R}^{N}$ can be expressed as
\[
\V{v} = \alpha \V{v}_1 + \V{w},\quad \V{w} = \sum_{i = 2}^N \alpha_i \V{v}_i .
\]
We have
\begin{align*}
& \V{v}^T \Big(\mu (M_p^{\circ})^{-\frac{1}{2}} D (M_p^{\circ})^{-\frac{1}{2}} +
(M_p^{\circ})^{-\frac{1}{2}}B^{\circ} A^{-1} (B^{\circ})^T(M_p^{\circ})^{-\frac{1}{2}}\Big) \V{v}
\\
& = \frac{\mu d_{11}}{|K_1|} \left ( \alpha \sqrt{ \frac{|K_1|}{|\Omega|}} + w_1 \right )^2 + \V{w}^T  (M_p^{\circ})^{-\frac{1}{2}}B^{\circ} A^{-1} (B^{\circ})^T(M_p^{\circ})^{-\frac{1}{2}} \V{w}
\\
& \ge \frac{\mu d_{11}}{|K_1|} \left ( \alpha \sqrt{ \frac{|K_1|}{|\Omega|}} + w_1 \right )^2 
+ \beta^2 \V{w}^T \V{w} ,
\end{align*}
where $w_1$ is the first component of $\V{w}$.
From this, we have 
\begin{align*}
    \lambda_1 
    &= \min_{\V{v}^T\V{v} = 1} \V{v}^T \Big(\mu (M_p^{\circ})^{-\frac{1}{2}} D (M_p^{\circ})^{-\frac{1}{2}} +
(M_p^{\circ})^{-\frac{1}{2}}B^{\circ} A^{-1} (B^{\circ})^T(M_p^{\circ})^{-\frac{1}{2}}\Big) \V{v}
\notag 
\\ \displaystyle
& \ge   \min_{\alpha^2 + \V{w}^T \V{w} = 1} \frac{\mu d_{11}}{|K_1|} \left (\alpha \sqrt{ \frac{|K_1|}{|\Omega|}} + w_1\right )^2 + 
\beta^2 \V{w}^T \V{w} .
\end{align*}
We enlarge the search domain and get
\begin{align}
\nonumber
    \lambda_1
  \ge   \min\limits_{\substack{0 \le \alpha^2 \le 1\\[0.05in] 0 \le w_1^2 \le 1-\alpha^2}}
  \frac{\mu d_{11}}{|K_1|} \left (\alpha \sqrt{ \frac{|K_1|}{|\Omega|}} + w_1\right )^2 + 
\beta^2 (1-\alpha^2) .
\end{align}
The only critical point within the search domain is $(\alpha,w_1) = (0,0)$, and the resulting value
of the objective function is $\beta^2$.
Comparing the values of the objective function on the boundary of the search domain,
we find the minimum value ${\mu d_{11}}/{|\Omega|}$ is reached at $w_1 = 0$ and $\alpha = 1$.
\end{proof}

Now, we turn our attention back to (\ref{GMRES-residual-3}) to derive the bound for the residual of GMRES.

\begin{pro}
    \label{GMRES_conv}
Assume that $ \mu \frac{d_{11}}{|K_1|} \ll \beta^2$ is satisfied.
Then, the residual of GMRES applied to the preconditioned system $\mathcal{P}_{t}^{-1} \mathcal{A}$ is bounded by
\begin{align}
\frac{\| \V{r}_k\|}{\| \V{r}_0\|} \; {\stackrel{<}{\sim}} \; \frac{2 |\Omega|(d+1)\left (d+2+\Big (\frac{d\, \lambda_{\max} (M_p^{\circ})}{\lambda_{\min} (A)}\Big )^\frac{1}{2}\right )}{\mu d_{11}} \left ( \frac{\sqrt{d}-\beta}{\sqrt{d}+\beta}\right )^{k-2} .
\label{GMRES-residual-5}
\end{align}
\end{pro}

\begin{proof}
The minmax problem in (\ref{GMRES-residual-3}) can be solved using shifted Chebyshev polynomials (e.g., see \cite[Pages 50-52]{Greenbaum-1997}). We have
\begin{align}
    \min\limits_{\substack{p \in \mathbb{P}_{k-1}\\ p(0) = 1}} \| p(\hat{S}^{-1}S) \|
    & = 
     \min\limits_{\substack{p \in \mathbb{P}_{k-1}\\ p(0) = 1}} 
    \max_{i = 1,..., N} |p(\lambda_i)|
    \nonumber
    \\
    &\le
    \min\limits_{\substack{p \in \mathbb{P}_{k-2}\\ p(0) = 1}} 
    \max_{i = 2, ..., N} |\frac{(\lambda_i - \lambda_1)}{\lambda_1} p(\lambda_i) | \nonumber
    \\
    &\le \frac{d+\mu \frac{d_{11}}{|K_1|}-\lambda_1}{\lambda_1}
    \min\limits_{\substack{p \in \mathbb{P}_{k-2}\\ p(0) = 1}} 
    \max_{\beta^2-\mu \frac{d_{11}}{|K_1|} \le \lambda \le d + \mu \frac{d_{11}}{|K_1|}} | p(\lambda) | \nonumber
    \\
    &\le 
    2 \frac{d+\mu \frac{d_{11}}{|K_1|}-\lambda_1}{\lambda_1} \left (\frac{\sqrt{d+\mu \frac{d_{11}}{|K_1|}} - \sqrt{\beta^2 - \mu \frac{d_{11}}{|K_1|}}}{\sqrt{d+\mu \frac{d_{11}}{|K_1|}} + \sqrt{\beta^2 - \mu \frac{d_{11}}{|K_1|}}}\right )^{k-2} .
    \notag
\end{align}
Combining this with (\ref{GMRES-residual-3}) and using the lower bound for $\lambda_1$, we obtain
\begin{align}
    \frac{\| \V{r}_k\|}{\| \V{r}_0\|}
    & \le 2\left (1+d+\mu \frac{d_{11}}{|K_1|} + \Big (\frac{d\, \lambda_{\max} (M_p^{\circ})}{\lambda_{\min} (A)}\Big )^\frac{1}{2}\right ) 
    \cdot  \frac{d+\mu \frac{d_{11}}{|K_1|}-\lambda_1}{\lambda_1} 
    \notag \\
    & \qquad \qquad \qquad \cdot \left (\frac{\sqrt{d+\mu \frac{d_{11}}{|K_1|}} - \sqrt{\beta^2 - \mu \frac{d_{11}}{|K_1|}}}{\sqrt{d+\mu \frac{d_{11}}{|K_1|}} + \sqrt{\beta^2 - \mu \frac{d_{11}}{|K_1|}}}\right )^{k-2} 
    \notag
    \\
    & \le 2\left  (1+d+\mu \frac{d_{11}}{|K_1|} + \Big (\frac{d\, \lambda_{\max} (M_p^{\circ})}{\lambda_{\min} (A)}\Big )^\frac{1}{2}\right ) 
    \cdot  \frac{d+\mu \frac{d_{11}}{|K_1|}}{\mu \frac{d_{11}}{|\Omega|}}
    \notag \\
    & \qquad \qquad \qquad \cdot \left (\frac{\sqrt{d+\mu \frac{d_{11}}{|K_1|}} - \sqrt{\beta^2 - \mu \frac{d_{11}}{|K_1|}}}{\sqrt{d+\mu \frac{d_{11}}{|K_1|}} + \sqrt{\beta^2 - \mu \frac{d_{11}}{|K_1|}}}\right )^{k-2} .
    \notag
\end{align}
If $\mu \frac{d_{11}}{|K_1|} \ll \beta^2$ is satisfied,
the above estimate can be simplified into (\ref{GMRES-residual-5}).
\end{proof}

This estimate indicates that the asymptotic convergence factor of GMRES applied to the regularized system (\ref{regularized_scheme})
with the preconditioner (\ref{PrecondP}) is almost independent of $h$ and $\mu$.
Moreover, recalling that $\lambda_{\max} (M_p^{\circ})/\lambda_{\min} (A)$ is bounded above by a constant for a quasi-uniform mesh,
from  (\ref{GMRES-residual-5}) we see that the constant in the above asymptotic error bound depends on
$\mu$ and the choice of $d_{11}$. 
Similar to the asymptotic error constant for MINRES with a block diagonal preconditioner, the number of GMRES iterations is proportional to the logarithm of the asymptotic error constant, i.e.,  $\log (\mu d_{11})$.
This finding is consistent with the analysis of Campbell et al. \cite{Campbell-1996}
that shows that if the eigenvalues of the coefficient matrix consist of a single cluster plus outliers,
then the convergence factor of GMRES is bounded by the cluster radius, while the asymptotic error constant reflects
the non-normality of the coefficient matrix and the distance of the outliers from the cluster. 

It is interesting to point out that the bounds for MINES and GMRES, (\ref{MINRES_res}) and (\ref{GMRES-residual-5}), are very similar. Particularly, they have almost the same convergence factor. These bounds also indicate that MINRES with a block diagonal
preconditioner may need twice as many iterations as GMRES with a block triangular preconditioner to reach a prescribed level of the residual. Although not directly comparable, it is known \cite[Theorem 8.2]{Elman-2014} that
GMRES with a block triangular preconditioner requires half as many iterations as when a block diagonal preconditioner is used.

\section{Numerical experiments}
\label{SEC:numerical}

In this section we present some two- and three-dimensional numerical results to demonstrate the performance of MINRES 
with the block diagonal \eqref{PrecondPd} and GMRES with block triangular Schur
complement preconditioner (\ref{PrecondP}) for the regularized system (\ref{scheme_matrix_2}).
We use MATLAB's function {\em minres} with $tol = 10^{-9}$ for 2D examples and  $tol = 10^{-8}$ for 3D examples with block diagonal preconditioners,  a maximum of 1000 iterations, and the zero vector as the initial guess.
For preconditioned systems with block triangular preconditioners, we use MATLAB's function {\em gmres} with $tol = 10^{-9}$ for 2D examples and  $tol = 10^{-8}$ for 3D examples with block triangular preconditioners, $restart = 30$, and the zero vector as the initial guess.
The implementation of block preconditioners requires the action of the inversion of the diagonal blocks.
The $(2,2)$-block is the mass matrix $M_{p}^{\circ}$ which is diagonal and its inversion is trivial.
The leading block $A$ is the WG approximation of the Laplacian operator. The conjugate gradient method preconditioned
with incomplete Cholesky decomposition is used for solving linear systems associated with $A$.
The incomplete Cholesky decomposition is carried out using MATLAB's function {\em ichol} with threshold
dropping and the drop tolerance is $10^{-3}$.
Triangular and tetrahedral meshes as shown in Fig.~\ref{Mesh} are used for the computation in two and three dimensions.

\begin{figure}
    \centering
  \subfigure[A triangular mesh] {\includegraphics[width=0.3\linewidth]{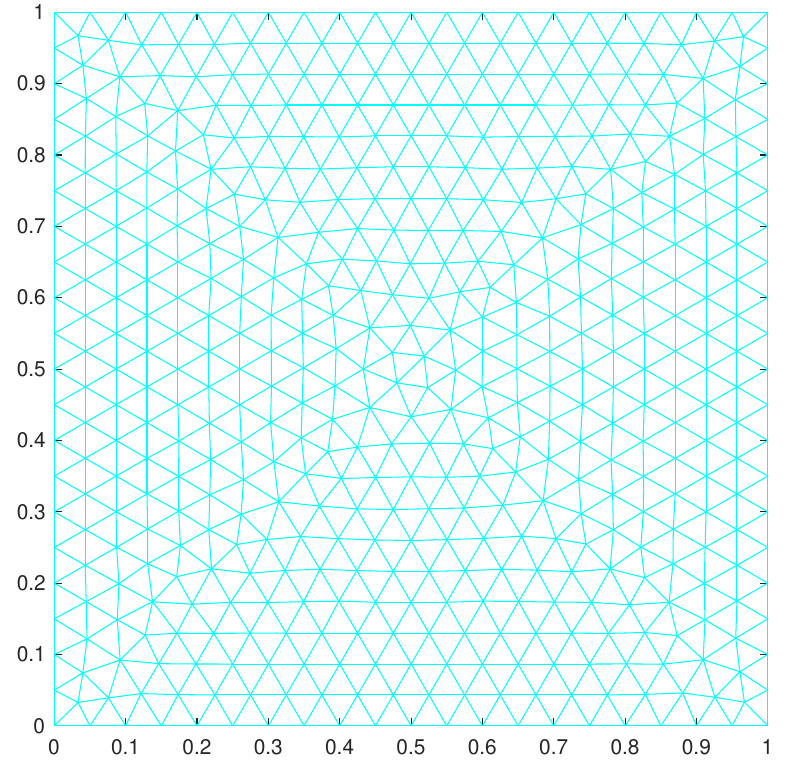}}
  \subfigure[A tetrahedral mesh] {\includegraphics[width=0.3\linewidth]{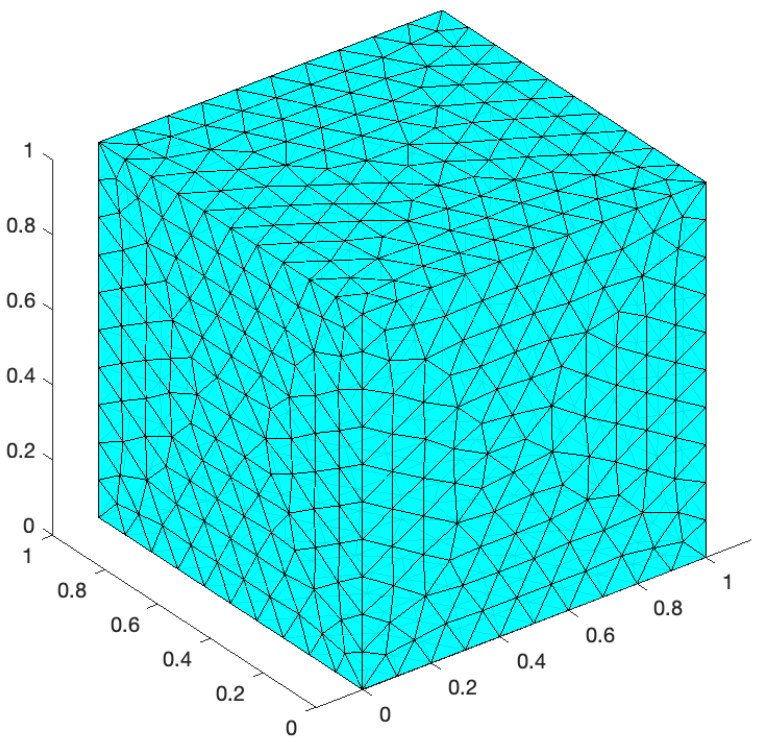}}
    \caption{Examples of meshes used for the computation in two and three dimensions.}
    \label{Mesh}
\end{figure}

\subsection{The two-dimensional example}
This two-dimensional (2D) example is taken from \cite{Mu.2020}, where
$\Omega = (0,1)^2$, the exact solutions are 
\begin{align*}
\V{u} = 
\begin{bmatrix}
    -e^x(y \cos(y)+\sin(y)) \\
    e^x y \sin(y)
\end{bmatrix}, \quad
p = 2 e^x \sin(y),
\end{align*}
and the right-hand side function is
\begin{align*}
\V{f} = 
\begin{bmatrix}
    2(1-\mu) e^x  \sin(y) \\
    2(1-\mu) e^x \cos(y)
\end{bmatrix}.
\end{align*}
We test the performance of the preconditioner with two values of viscosity, $\mu = 1$ and $10^{-4}$.

\subsubsection{Results for $\mathcal{P}_d^{-1} \mathcal{A}$}

The regularization in (\ref{regularized_scheme}) maintains the same solution as the original scheme
and does not affect the optimal convergence of the WG approximation. Since the corresponding numerical results
have been presented in \cite{WANG202290}, we focus here on the performance of the preconditioner $\mathcal{P}_d$ \eqref{PrecondPd}
for the preconditioned system $\mathcal{P}_d^{-1} \mathcal{A}$.
We examine the two choices $d_{11} = 1$ and $|K_1|$ that have been discussed in Section~\ref{sec::regularization-diagonal}.

Table~\ref{Pdiag_2D} shows the number of MINRES iterations to reach the required tolerance.
The number of iterations remains relatively small and oscillates slightly but has no significant change
when the mesh is refined and $\mu$ changes from $1$ to $10^{-4}$. 
Moreover, the choices $d_{11} = 1$ and $|K_1|$ seem to lead to almost the same performance.
MINRES with the preconditioner (\ref{PrecondPd}) seems to perform better for this example than what
indicated by the worst-case-scenario estimate (\ref{MINRES_res}) where the asymptotic constant is proportional to $1/\mu$. 
It is worth reporting that MINRES without preconditioning would take more than 1000 iterations to reach convergence.
Comparing this with Table~\ref{Pdiag_2D}, we conclude that the preconditioner (\ref{PrecondPd}) is effective.

\begin{table}[h]
    \centering
        \caption{The 2D Example: The number of MINRES iterations required to reach convergence for preconditioned systems $\mathcal{P}_d^{-1} \mathcal{A}$, with $d_{11} = 1$ and $|K_1|$ and $\mu = 1$ and $10^{-4}$.}
    \begin{tabular}{|c|c|c|c|c|c|c|}
        \hline
      $d_{11}$ & \diagbox{$\mu$}{$N$} & 232 & 918 & 3680 & 14728 & 58608 \\ \hline
       \multirow{ 2}{*}{1} & $1$ & 68 & 76 & 86 & 100 & 121 \\ \cline{2-7}
        & $10^{-4}$ & 66  & 75 & 81 & 90 & 102 \\ \hline 
       \multirow{ 2}{*}{$|K_1|$} & $1$ & 63 & 71 & 77 & 83 & 91\\ \cline{2-7}
        & $10^{-4}$  & 76 & 87 & 94& 110 & 116\\ \hline 
    \end{tabular}
    \label{Pdiag_2D}
\end{table}

\subsubsection{Results for $\mathcal{P}_t^{-1} \mathcal{A}$}

Table~\ref{P2step-2d} shows the number of GMRES iterations to reach the required tolerance.
Similar to the results for MINRES, the number of iterations remains relatively small,
with slight oscillations but no significant change, as the mesh is refined and $\mu$ varies from $1$ to $10^{-4}$.  
Moreover, the choices $d_{11} = 1$ and $|K_1|$ seem to lead to almost the same performance.
GMRES without preconditioning would take more than 30,000 iterations to reach convergence.
This demonstrates that the preconditioner (\ref{PrecondP}) is effective.

Comparing Tables~\ref{Pdiag_2D} and \ref{P2step-2d}, one can see that the number of MINRES iterations
is approximately double that of GMRES. This is consistent with the theoretical analysis
in Sections~\ref{sec::regularization-diagonal} and \ref{sec::regularization-triagle}.

\begin{table}[h]
    \centering
        \caption{The 2D Example: The number of GMRES iterations required to reach convergence for preconditioned systems $\mathcal{P}_t^{-1} \mathcal{A}$, with $d_{11} = 1$ and $|K_1|$ and $\mu = 1$ and $10^{-4}$.}
    \begin{tabular}{|c|c|c|c|c|c|c|}
        \hline
      $d_{11}$ & \diagbox{$\mu$}{$N$} & 232 & 918 & 3680 & 14728 & 58608 \\ \hline
       \multirow{ 2}{*}{1} & $1$ & 30 & 36 & 39 & 55 & 61 \\ \cline{2-7}
        & $10^{-4}$ & 33  & 38 & 53 & 58 & 66 \\ \hline 
       \multirow{ 2}{*}{$|K_1|$} & $1$ & 30 & 36 & 38 & 56 & 59\\ \cline{2-7}
        & $10^{-4}$  & 55 & 54 & 51& 52 & 56\\ \hline 
    \end{tabular}
    \label{P2step-2d}
\end{table}

\subsection{The three-dimensional example}
This three-dimensional (3D) example is adopted from \textit{deal.II} \cite{dealii} \texttt{step-56}
where $\Omega = (0,1)^3$, the exact solutions are 
\begin{align*}
\V{u} = 
\begin{bmatrix}
    2 \sin(\pi x) \\
    -\pi y \cos(\pi x) \\
    -\pi z \cos(\pi x)
\end{bmatrix}, \quad
p = \sin(\pi x) \cos(\pi y) \sin(\pi z),
\end{align*}
and the right-hand side function is
\begin{align*}
\V{f} = 
\begin{bmatrix}
    2 \mu \pi^2 \sin(\pi x) + \pi \cos(\pi x) \cos(\pi y) \sin(\pi z) \\
    -\mu \pi^3 y \cos(\pi x) - \pi \sin(\pi y) \sin(\pi x) \sin(\pi z)\\
    -\mu \pi^3 z \cos(\pi x) + \pi \sin(\pi x) \cos(\pi y) \cos(\pi z)
\end{bmatrix}.
\end{align*}

Tables~\ref{Pdiag_3d} and \ref{P2-3d-cv} show the number of MINRES iterations for the preconditioned system $\mathcal{P}_d^{-1} \mathcal{A}$ and the number of GMRES iterations for the preconditioned system $\mathcal{P}_t^{-1} \mathcal{A}$,
respectively, for $\mu = 1$ and $10^{-4}$.
The numbers remain relatively small but oscillate slightly over different mesh size and
different values of $\mu$. The oscillations may be a reflection on the fact that the asymptotic error
constant depends weakly on $h$ and $\mu$ for both cases.

\begin{table}[h]
    \centering
        \caption{The 3D Example: The number of MINRES iterations required to reach convergence for preconditioned systems $\mathcal{P}_d^{-1} \mathcal{A}$ with $d_{11} = 1$ and $|K_1|$ and $\mu = 1$ and $10^{-4}$.}
    \begin{tabular}{|c|c|c|c|c|c|c|}
        \hline
      $d_{11}$ & \diagbox{$\mu$}{$N$} & 4046 & 7915 & 32724 & 112078 & 266555\\ \hline
       \multirow{ 2}{*}{1} & $1$ & 110 &118 & 83 & 91 & 94 \\ \cline{2-7}
        & $10^{-4}$ &149  & 100& 105 & 123 &139 \\ \hline 
       \multirow{ 2}{*}{$|K_1|$} & $1$ &62 &96 &64 & 66 & 68\\ \cline{2-7}
        & $10^{-4}$  &62 & 62& 70& 76& 78\\ \hline 
    \end{tabular}
    \label{Pdiag_3d}
\end{table}





\begin{table}[h]
    \centering
        \caption{The 3D Example: The number of GMRES iterations required to reach convergence for preconditioned systems $\mathcal{P}_t^{-1} \mathcal{A}$ with $d_{11} = 1$ and $|K_1|$ and $\mu = 1$ and $10^{-4}$.}
    \begin{tabular}{|c|c|c|c|c|c|c|}
        \hline
      $d_{11}$ & \diagbox{$\mu$}{$N$} & 4046 & 7915 & 32724 & 112078 & 266555\\ \hline
       \multirow{ 2}{*}{1} & $1$ & 59 &64 & 57 & 57 & 61 \\ \cline{2-7}
        & $10^{-4}$ &58  & 58& 63 & 63 &67 \\ \hline 
       \multirow{ 2}{*}{$|K_1|$} & $1$ &55 &61 &56 & 58 & 61\\ \cline{2-7}
        & $10^{-4}$  &35 & 35& 37& 38& 38\\ \hline 
    \end{tabular}
    \label{P2-3d-cv}
\end{table}





\section{Conclusions}
\label{SEC:conclusions}

In the previous sections we have studied the convergence of the MINRES and GMRES iterative solution of
the lowest-order weak Galerkin finite element approximation of Stokes problems.
The resulting saddle point system \eqref{scheme_matrix} is singular,
with one-rank deficiency in the (1,2) (and (2,1)) block of the coefficient matrix.
To address the singularity of the system, we have applied a commonly used local technique by specifying the value zero of the pressure at the barycenter of the first element.
In Section~\ref{SEC:formulation} we have analytically proved the nonsingularity of the regularized system (\ref{scheme_matrix_2}) and obtained the bounds for the Schur complement.

We have considered block diagonal and triangular Schur complement preconditioners
for the iterative solution of the regularized system.
In Section~\ref{sec::regularization-diagonal}, we have studied the block diagonal Schur complement preconditioner
\eqref{PrecondPd} and established bounds for the eigenvalues of the preconditioned system (see Lemma~\ref{lem:eigen_bound_diag})
and for the residual of MINRES applied to the preconditioned system (cf. Proposition~\ref{MINRES_conv}). 
These bounds show that the convergence factor of MINRES is nearly independent of $\mu$ and $h$
while the number of iterations required to reach convergence depends logarithmically on these parameters.

In Section~\ref{sec::regularization-triagle}, we have studied the block triangular Schur complement preconditioner
(\ref{PrecondP}). For this case, the preconditioned system is non longer diagonalizable. As a consequence, we need to use
GMRES for the iterative solution of the preconditioned system. Moreover, the spectral analysis is insufficient to determine
the convergence of GMRES. Lemmas~\ref{lem:gmres-conv-lower} and \ref{lem:gmres-conv-upper}
developed in Appendix~\ref{appendix:gmres convergence} have been used to analyze the convergence of GMRES.
More specifically, for the preconditioner (\ref{PrecondP}) this has been done through estimating
$\| A^{-1} (B^{\circ})^T \|$, $\| \hat{S}^{-1}S\|$, and $\min_{p} \| p(\hat{S}^{-1}S) \|$.
The bounds for $\| A^{-1} (B^{\circ})^T \|$ and $\| \hat{S}^{-1}S\|$ are given in (\ref{bound2}) and Lemma~\ref{lem:S-2}.
The bounds for the eigenvalues of $\hat{S}^{-1}S$ are given in Lemma~\ref{eigen_bound} and those for the residual of
GMRES applied to the preconditioned system are given in Proposition~\ref{GMRES_conv}. 
These bounds show that, like MINRES, the convergence factor of GMRES is almost independent of $\mu$ and $h$ but
the number of GMRES iterations required to reach a prescribed level of residual depends on the parameters logarithmically.

The numerical results in two and three dimensions presented in Section~\ref{SEC:numerical} have confirmed that the block diagonal \eqref{PrecondPd} and block triangular (\ref{PrecondP}) Schur complement preconditioners are effective for the regularized system (\ref{scheme_matrix_2}). These preconditioners use the exact leading diagonal block $A$ and an inexact approximation to
the Schur complement. The action of the inversion of $A$ can be carried out efficiently using a direct sparse solver
or an iterative solver with preconditioning. Moreover, an inexact approximation of $A$ can be used for
the preconditioners in practical computation. 


\section*{Acknowledgments}

W.~Huang was supported in part by the Air Force Office of Scientific Research (AFOSR) grant FA9550-23-1-0571
and the Simons Foundation grant MPS-TSM-00002397.

\appendix

\section{Convergence analysis of GMRES for saddle point problems with block triangular Schur complement preconditioning}
\label{appendix:gmres convergence}

Consider saddle point systems in the general form
\begin{equation}
\label{SPP-1}
\begin{bmatrix} A & B^T \\ C & - D\end{bmatrix} \begin{bmatrix} \V{u} \\ \V{p} \end{bmatrix} =
\begin{bmatrix} \V{b}_1 \\ \V{b}_2 \end{bmatrix},
\qquad \mathcal{A} = \begin{bmatrix} A & B^T \\ C & - D\end{bmatrix} ,
\end{equation}
where $A$ is assumed to be nonsingular. The system is not assumed to be symmetric here, neither does $B$ or $C$
have full rank nor is the Schur complement $- D - C A^{-1} B^T$ nonsingular.
For notational convenience and without causing confusion,
we denote $S = D  + CA^{-1}B^T$ and refer it as the Schur complement instead. 

A commonly used strategy for iterative solutions is to employ a Krylov subspace method with a preconditioner.
Two of widely used preconditioners are (inexact) block diagonal and block triangular Schur
complement preconditioners,
\begin{align}
\label{block-schur-precond}
    \mathcal{P}_{d^{\pm}} = \begin{bmatrix}
        \hat{A} & 0 \\ 0 & \pm \hat{S}
    \end{bmatrix},\quad 
    \mathcal{P}_{tL^{\pm}} = \begin{bmatrix}
        \hat{A} & 0 \\ C & \pm \hat{S}
    \end{bmatrix},\quad 
    \mathcal{P}_{tU^{\pm}} = \begin{bmatrix}
        \hat{A} & B^T \\ 0 & \pm \hat{S}
    \end{bmatrix},
\end{align}
where $\hat{A}$ and $\hat{S}$ are approximations to $A$ and $S$, respectively.
Generally speaking, the preconditioned systems associated with these preconditioners
are not diagonalizable.
The exceptions include $\mathcal{P}_{d^{\pm}}^{-1} \mathcal{A}$,
$\mathcal{P}_{tL^{+}}^{-1} \mathcal{A}$, and $\mathcal{P}_{tU^{+}}^{-1} \mathcal{A}$ with exact $\hat{A} = A$ 
and $\hat{S} = S$ (e.g., see \cite{MurphyGolubWathen_SISC_2000}) and
$\mathcal{P}_{d^{+}}^{-1} \mathcal{A}$ for general $\hat{A}$ and $\hat{S}$ assuming that
$\mathcal{A}$ is symmetric and $\mathcal{P}_{d^{+}}$ is symmetric and positive definite (SPD).
For the latter case, the minimal residual method (MINRES) can be applied and the convergence analysis can be carried
out using spectral analysis \cite{Elman-2014}.
For other situations, particularly for $\mathcal{P}_{tL^{\pm}}$ and $\mathcal{P}_{tU^{\pm}}$
with inexact $\hat{A}$ and/or inexact $\hat{S}$, the preconditioned systems are not diagonalizable.
For these systems, the convergence of GMRES (or any other suitable Krylov subspace method) is
difficult to analyze in general since the spectral information is insufficient in determining the convergence behavior.
As a matter of fact, limited analysis work has been done with block triangular preconditioners
$\mathcal{P}_{tL^{\pm}}$ and $\mathcal{P}_{tU^{\pm}}$.
For symmetric saddle point systems,
Bramble and Pasciak \cite{Bramble_MathComp_1988} considered the lower block triangular preconditioner $\mathcal{P}_{tL^{-1}}$
and showed that $\mathcal{P}_{tL^{-1}}^{-1} \mathcal{A}$ is SPD in the inner product associated with
the matrix $\text{diag}(A-\hat{A}, \hat{S})$ (which is assumed to be SPD), the corresponding preconditioned
system can be solved using the conjugate gradient method,  and the convergence can be analyzed accordingly.

We consider the block triangular preconditioners $\mathcal{P}_{tL^{\pm}}$ and $\mathcal{P}_{tU^{\pm}}$
with exact $\hat{A} = A$ but inexact Schur complement $\hat{S}$. For convenience, we denote these inexact block
triangular Schur complement preconditioners as
\begin{align}
\label{SPP-4}
    \mathcal{P}_{L^{\pm}} = \begin{bmatrix}
        A & 0 \\ C & \pm \hat{S}
    \end{bmatrix},
    \quad 
    \mathcal{P}_{U^{\pm}} = \begin{bmatrix}
    A & B^{T} \\ 0 & \pm \hat{S}
    \end{bmatrix},
\end{align}
where $\hat{S}$ is an approximation to $S$ and is assumed to be nonsingular.
The corresponding preconditioned systems read as
\begin{align}
    \label{SPP-2}
    \mathcal{P}_{L^{\pm}}^{-1} \mathcal{A} \begin{bmatrix} \V{u} \\ \V{p} \end{bmatrix}
    = \mathcal{P}_{L^{\pm}}^{-1} \begin{bmatrix} \V{b}_1 \\ \V{b}_2 \end{bmatrix}, \qquad 
    \mathcal{A} \mathcal{P}_{U^{\pm}}^{-1} \left ( \mathcal{P}_{U^{\pm}} \begin{bmatrix} \V{u} \\ \V{p} \end{bmatrix} \right ) 
    = \begin{bmatrix} \V{b}_1 \\ \V{b}_2 \end{bmatrix} .
\end{align}
It can be verified that
\begin{align}
    \label{SPP-3}
    \mathcal{P}_{L^{\pm}}^{-1} \mathcal{A} = \begin{bmatrix} I & A^{-1} B^T \\ 0 & \mp \hat{S}^{-1} S\end{bmatrix},
    \qquad  \mathcal{A} \mathcal{P}_{U^{\pm}}^{-1} = \begin{bmatrix} I & 0 \\ CA^{-1} & \mp  S \hat{S}^{-1}\end{bmatrix} .
\end{align}
For many applications, $\mathcal{P}_{L^{+}}^{-1} \mathcal{A}$ and
$\mathcal{P}_{U^{+}}^{-1} \mathcal{A}$ have eigenvalues on both sides of the imaginary axis while
$\mathcal{P}_{L^{-}}^{-1} \mathcal{A}$ and $\mathcal{P}_{U^{-}}^{-1} \mathcal{A}$ have eigenvalues only
on the right side of the imaginary axis. Moreover, when $\hat{S} \neq S$, these matrices are not diagonalizable in general.
Despite this, the following two lemmas provide an estimate on the residual of GMRES in terms of
$\hat{S}^{-1} S$ or $ S \hat{S}^{-1}$.

\begin{lem}
\label{lem:gmres-conv-lower}
    The residual of GMRES applied to the preconditioned system $\mathcal{P}_{L^{\pm}}^{-1} \mathcal{A}$ is bounded by
    \begin{align}
        \label{lem:lowerP_res-1}
        \frac{\| \V{r}_k \|}{\|\V{r}_0\|} \le
        (1+\|A^{-1}B^T\| + \| \hat{S}^{-1} S \|) \min\limits_{\substack{p \in \mathbb{P}_{k-1}\\ p(0) = 1}} \| p(\hat{S}^{-1} S) \| ,
        \end{align}
where $\mathbb{P}_{k-1}$ denotes the set of polynomials of degree up to $k-1$.
\end{lem}

\begin{proof}
For the residual of GMRES \cite{GMRES-1986} for $\mathcal{P}_{L^{\pm}}^{-1} \mathcal{A}$, we have
\begin{align}
\| \V{r}_k\| = \min\limits_{\substack{p \in \mathbb{P}_{k}\\ p(0) = 1}} \| p (\mathcal{P}_{L^{\pm}}^{-1} \mathcal{A})  r_0 \|
\le \min\limits_{\substack{p \in \mathbb{P}_{k-1}\\ p(0) = 1}} \| (I - \mathcal{P}_{L^{\pm}}^{-1} \mathcal{A})
\;p (\mathcal{P}_{L^{\pm}}^{-1} \mathcal{A})\| \|  r_0 \|.
\label{lem_P1_2}
\end{align}
From Cauchy's integral formula (e.g., see \cite[Equation (9.2.8)]{MR3024913}), we have
\begin{align}
    (I - \mathcal{P}_{L^{\pm}}^{-1} \mathcal{A}) \; p (\mathcal{P}_{L^{\pm}}^{-1} \mathcal{A}) =
    \frac{1}{2\pi i} \int_{\gamma} (1-z) \; p(z) (zI - \mathcal{P}_{L^{\pm}}^{-1} \mathcal{A}) ^{-1} dz,
    \label{lem_P1_3}
\end{align}
where $i^2 = -1$ and $\gamma$ is a closed contour enclosing the spectrum of $\mathcal{P}_{t}^{-1} \mathcal{A}$ on the complex plane.
From \eqref{SPP-3}, we have
\begin{align*}
(zI - \mathcal{P}_{L^{\pm}}^{-1} \mathcal{A}) ^{-1} =
    \begin{bmatrix}
        (z-1)^{-1} I & (z-1)^{-1} A^{-1} B^T (zI \pm \hat{S}^{-1}S)^{-1} \\
        0 & (zI \pm \hat{S}^{-1}S)^{-1}
    \end{bmatrix}.
\end{align*}
Inserting this into (\ref{lem_P1_3}) gives
\begin{align*}
    (I - \mathcal{P}_{L^{\pm}}^{-1} \mathcal{A}) \; p (\mathcal{P}_{L^{\pm}}^{-1} \mathcal{A}) =
    \begin{bmatrix}
        \frac{-1}{2 \pi i} \int_{\gamma} p(z)I dz & 
        - \frac{A^{-1}B^T}{2 \pi i} \int_{\gamma} p(z) (zI \pm \hat{S^{-1}} S)^{-1} dz \\
        0 &  \frac{1}{2 \pi i} \int_{\gamma} (1-z) p(z) (zI  \pm \hat{S^{-1}} S)^{-1} dz
    \end{bmatrix}.
\end{align*}
Notice that $\int_{\gamma} p(z) dz = 0$ since $p(z)$ is analytic. Using Cauchy's integral formula again for other entries, we get
\begin{align*}
     (I - \mathcal{P}_{L^{\pm}}^{-1} \mathcal{A}) \; p (\mathcal{P}_{L^{\pm}}^{-1} \mathcal{A}) =
     \begin{bmatrix}
         0 & - A^{-1} B^T p(\mp \hat{S}^{-1}S) \\
         0 & (I \pm \hat{S}^{-1} S) p(\mp \hat{S}^{-1}S)
     \end{bmatrix}.
\end{align*}
Combining this with (\ref{lem_P1_2}), we have
\begin{align*}
    \frac{\| \V{r}_k \|}{\|\V{r}_0 \|} 
    &\le \min \limits_{\substack{p \in \mathbb{P}_{k-1}\\ p(0) = 1}} \left (\|A^{-1} B^T p(\mp \hat{S}^{-1}S)\|^2  + \|(I \pm \hat{S}^{-1} S) p(\mp \hat{S}^{-1}S) \|^2\right )^{\frac{1}{2}}
    \\
    & \le (1 + \|A^{-1} B^T\| + \|\hat{S}^{-1}S\| ) \min \limits_{\substack{p \in \mathbb{P}_{k-1}\\ p(0) = 1}} 
     \| p( \hat{S}^{-1}S )\| ,
\end{align*}
which gives (\ref{lem:lowerP_res-1}).
\end{proof}

Similarly, we can prove the following lemma.

\begin{lem}
\label{lem:gmres-conv-upper}
    The residual of GMRES applied to the preconditioned system $\mathcal{A} \mathcal{P}_{U^{\pm}}^{-1} $ is bounded by
    \begin{align}
        \frac{\| \V{r}_k \|}{\|\V{r}_0\|} \le
        (1+\|C A^{-1}\| + \|  S \hat{S}^{-1} \|) \min\limits_{\substack{p \in \mathbb{P}_{k-1}\\ p(0) = 1}} \| p(S \hat{S}^{-1}) \| .
        \end{align}
\end{lem}

From the above two lemmas, we can make the following observations.
\begin{itemize}
\item The lemmas show that, to estimate the residual of GMRES for the preconditioned saddle point systems (\ref{SPP-2}),
    basically we just need to estimate
    \[
    \min\limits_{\substack{p \in \mathbb{P}_{k-1}\\ p(0) = 1}} \| p(\hat{S}^{-1} S) \| \quad \text{ or }\quad 
    \min\limits_{\substack{p \in \mathbb{P}_{k-1}\\ p(0) = 1}} \| p(S \hat{S}^{-1}) \|,
    \]
    which reflects the performance of GMRES for the smaller system $\hat{S}^{-1} S$ or $S \hat{S}^{-1}$.
    Analyzing the latter is a much easier
    task for many applications. For example, for symmetric saddle point problems, $\hat{S}^{-1} S$ or $S \hat{S}^{-1}$
    is often diagonalizable. In this case, spectral analysis is sufficient.
\item The convergence factor of GMRES for the preconditioned saddle point problems in (\ref{SPP-2}) is determined
    by the convergence factor of GMRES applied to smaller systems associated with $\hat{S}^{-1} S$ or $S \hat{S}^{-1}$.
\item The asymptotic error constant of GMRES for the preconditioned saddle point problems in (\ref{SPP-2})
    depends on $\| \hat{S}^{-1} S\|$, $\|A^{-1} B^T\|$, and $\|C A^{-1}\|$, which reflects the effects
    of the departure from the normality of the preconditioned problems.
\item The four block triangular Schur complement preconditioners in (\ref{SPP-4}) lead to similar bounds
    for the residual of GMRES and are expected to perform similarly.
\end{itemize}



\end{document}